\documentclass[11pt,a4paper,reqno]{amsart}

 \usepackage{geometry,dirtytalk,cases}
 \usepackage{xcolor}
\usepackage{amsmath,amssymb,amsthm,esint,empheq,etoolbox}

\title{Another look at elliptic homogenization} 
\author{Andrea Braides, Giuseppe Cosma Brusca, and Davide Donati}
\address{SISSA, via Bonomea 265, Trieste, Italy}

\theoremstyle{plain}
\begingroup
\newtheorem{theorem}{Theorem}
\newtheorem{lemma}[theorem]{Lemma}

\endgroup

\theoremstyle{definition}
\begingroup

\newtheorem{remark}[theorem]{Remark}

\endgroup

\newcommand\e{\varepsilon}

\newcommand\R{\mathbb{R}}    
\newcommand\Rd{\mathbb{R}^d}
\newcommand\Om{\Omega}

\newcommand\be{\begin{equation}}
\newcommand\ee{\end{equation}}
\newcommand\Omid{\Delta_i^{2\varepsilon}}
\newcommand\Omi{\Delta_i^{\varepsilon}}

\newcommand\Hd[1]{\mathcal{H}^{d-1}(#1)}

\theoremstyle{remark}
\begingroup
\endgroup

\begin{document}

\maketitle
\begin{abstract} We consider the limit of sequences of normalized $(s,2)$-Gagliardo seminorms with an oscillating coefficient as $s\to 1$. In a seminal paper by Bourgain, Brezis and Mironescu (subsequently extended by Ponce) it is proven that if the coefficient is constant then this sequence $\Gamma$-converges to a multiple of the Dirichlet integral. Here we prove that, if we denote by $\e$ the scale of the oscillations and we assume that $1-s<\!<\e^2$, this sequence converges to the homogenized functional formally obtained by separating the effects of $s$ and $\e$; that is, by the homogenization as $\e\to 0$ of the Dirichlet integral with oscillating coefficient obtained by formally letting $s\to 1$ first.
\smallskip



\end{abstract}

\section{Introduction}
In their seminal paper \cite{bbm} Bourgain, Brezis and Mironescu have studied the asymptotic behaviour of Gagliardo seminorms $[u]_{W^{s,p}(\Omega)}$ as $s\to 1$, and in particular, in the case $p=2$, of the seminorm $[u]_{W^{s,2}(\Omega)}$ given by
\[
[u]_{W^{s,2}(\Omega)}:=\Bigl(\iint_{\Omega\times\Omega}\frac{|u(x)-u(y)|^2}{|x-y|^{d+2s}}\,dxdy\Bigr)^{\frac{1}{2}},
\]
where $\Omega$ a bounded open subset of $\R^d$, $s\in (0,1)$, and the fractional Sobolev space $W^{s,2}(\Omega)$ is defined as
$$
W^{s,2}(\Omega):=\{u\in L^2(\Omega): [u]_{W^{s,2}(\Omega)} < +\infty\}.
$$
Their results, subsequently extended by Ponce \cite{Ponce}, imply that the functionals 
$$
F_s(u)=(1-s)\iint_{\Omega\times\Omega}\frac{|u(x)-u(y)|^2}{|x-y|^{d+2s}}\,dxdy
$$
$\Gamma$-converge, with respect to the $L^2$-convergence, to the (multiple of the) Dirichlet integral
$$
\frac{\sigma_{d-1}}{2d}\int_\Omega |\nabla u|^2\,dx,
$$
where $\sigma_{d-1}$ is the ${\mathcal H}^{d-1}$-dimensional measure of $S^{d-1}$,
with domain the usual Sobolev space $W^{1,2}(\Omega)$. Since the functionals are equi-coercive in $L^2(\Omega)$, by the properties of convergence of minima of $\Gamma$-convergence, minimum problems involving the Dirichlet integral can be approximated by problems involving Gagliardo seminorms, upon possibly suitably defining boundary-value problems if necessary. 

This fractional-Sobolev space approximation can be extended to other problems involving elliptic integrals, such as homogenization problems, the most classical of which is the asymptotic analysis of (isotropic) oscillating energies
\begin{equation}\label{oscillating}
\int_\Omega a\Bigl(\frac{x}{\e}\Bigr)|\nabla u|^2\,dx,
\end{equation}
where $a$ is a $1$-periodic function with $0<\alpha\le a(y)\le \beta<+\infty$. The $\Gamma$-limit as $\e\to 0$ of such energies is the {\em homogenized} functional
$$
F_{\rm hom}(u)=\int_\Omega\langle A_{\rm hom}\nabla u,\nabla u\rangle\,dx,
$$
where $A_{\rm hom}$ is a symmetric $d\times d$ matrix characterized by the {\em homogenization formula}
\begin{equation}\label{homfor}
\langle A_{\rm hom}z,z\rangle=\min\Bigl\{\int_{(0,1)^d} a(y)|z+\nabla \varphi(y)|^2dy: \varphi\  1\hbox{-periodic}\Bigr\}
\end{equation}
(see e.g.~\cite{BDF}).
Formally, by the result of Bourgain, Brezis and Mironescu, we can define an approximation of the homogenized functional using the energies
\begin{equation}\label{functional-0}
    F_{\varepsilon,s}(u)=(1-s)\iint_{\Omega\times \Omega} a\left(\frac{x}{\varepsilon}\right)\frac{\lvert u(x)-u(y)\rvert^2}{\lvert x-y\rvert^{d+2s}}dxdy,
\end{equation}
defined for $u\in W^{s,2}(\Omega)$ and $\e>0$. Indeed, upon supposing for simplicity that $a$ be continuous, we easily see that for fixed $\e>0$ the $\Gamma$-limit of $F_{\varepsilon,s}$ as $s\to 1^-$ is indeed 
$$
\frac{\sigma_{d-1}}{2d}\int_\Omega a\Bigl(\frac{x}{\e}\Bigr) |\nabla u|^2\,dx,
$$
which then $\Gamma$-converge to $\frac{\sigma_{d-1}}{2d} F_{\rm hom}$ as $\e\to0$.
Since all functionals are equi-coercive in $L^2$, we can use a diagonal argument (see \cite{DalMaso}) and deduce that there exist $s=s_\e$ such that $s\to 1^-$ as $\e\to0$ and the $\Gamma$-limit of $F_{\e,s}$ is still $\frac{\sigma_{d-1}}{2d} F_{\rm hom}$. 

In this paper we investigate the scales $s=s_\e$ at which this limit holds.
If we let both $\e,$ and $1-s$ simultaneously tend to $0$, a heuristic argument in order to argue what can be a critical scale is as follows. 
Let $z$ be fixed and let $\varphi$ be as in the  formula characterizing $\langle A_{\rm hom},z,z\rangle$.
We can use $u_\e(x)= \langle z,x\rangle+\e\varphi(x/\e)$ as test functions for the $\Gamma$-limit of $F_{\e,s}$.
We may suppose that $\varphi$ be twice continuous differentiable, so that, using the Taylor development of $\varphi$ at $x/\e$, we have
\begin{eqnarray*}
&&\hskip-2cm (1-s)\iint_{\Omega\times\Omega} a\Bigl(\frac{x}{\e}\Bigr)\frac{|u_\e(x)-u_\e(y)|^2}{|x-y|^{d+2s}}\,dxdy
\\
&=&(1-s)\iint_{\Omega\times\Omega} a\Bigl(\frac{x}{\e}\Bigr)\frac{|\langle z+\nabla\varphi(\frac{x}{\e}),x-y\rangle|^2+ O(\frac{1}{\e^2})}{|x-y|^{d+2s}}\,dxdy 
\\
&=& \frac{\sigma_{d-1}}{2d}\int_\Omega a\Bigl(\frac{x}{\e}\Bigr) \Big|z+\nabla\varphi\left(\frac{x}{\e}\right)\Big|^2dx + O\Big(\frac{1-s}{\e^2}\Big).
\end{eqnarray*}
This argument suggests that when
\begin{equation}\label{eq:isthemagicnumber}
1-s<\!< \e^2\qquad\qquad\hbox{({\em subcritical case})},
\end{equation}
 we may construct recovery sequences optimizing the oscillations of $a$, and actually is a proof of the upper bound in this case for the target function $u(x)=\langle z,x\rangle$. We will prove that under assumption \eqref{eq:isthemagicnumber} a {\em separation of the scales} $\e$ and $s$ occurs, and the limit is the one computed above by letting $s\to 1^-$ first and then $\e\to0$. The main argument in the proof of the lower bound is obtained by a discretization procedure (which as a byproduct also gives a different proof of the results in \cite{Ponce}) based on the use of Kuhn's decomposition \cite{Kuhn1960} and an integration argument on the set of orthogonal bases in $\mathbb R^d$ borrowed from a recent paper by Solci \cite{solci2023}.
 Thanks to \eqref{eq:isthemagicnumber}, we can then reduce the computation to studying the limit of oscillating energies \eqref{oscillating} on piecewise-affine interpolations at a scale much smaller than $\e$, for which the known homogenization result can be applied. This reduction is made possible by  a lemma (Lemma \ref{lemma:locality}), which allows to consider only points $x,y$ sufficiently close to each other in the computation of \eqref{functional-0}.

The cases other than subcritical are not dealt with here. They will require different, more complex, techniques and will be treated in future work. In order to give a hint of the necessity of different types of arguments in those cases, we can consider a similar type of non-local functionals treated in a general setting in 
 \cite{AABPT}, of the form 
\begin{equation}\label{functional-1}
    F_{\varepsilon,\delta}(u)=\frac{1}{\delta^d}\iint_{\Omega\times \Omega} \varrho\Big(\frac{|x-y|}{\delta}\Bigr)a\left(\frac{x}{\varepsilon}\right)\frac{\lvert u(x)-u(y)\rvert^2}{\lvert x-y\rvert^{2}}dxdy,
\end{equation}
where $\varrho$ is a suitably integrable positive kernel
and $\delta>0$ plays a similar role as $s$ above, in that it forces concentration as $\delta\to0$.
Such functionals are also studied in \cite{bbm} if the function $a$ is a constant, and still approximate the Dirichlet integrand as $\delta\to0$. The critical case for such functionals is $\delta\sim\e$, and if $\delta/\e\to \kappa\in(0,+\infty)$ the $\Gamma$-limit can be computed using a non-local-to-local homogenization procedure giving a local homogenized limit energy of the form
$$
F^\kappa_{\rm hom}(u)=\int_\Omega\langle A^\kappa_{\rm hom}\nabla u,\nabla u\rangle\,dx,
$$
where $A^\kappa_{\rm hom}$ is now characterized by a {\em non-local homogenization formula}
\begin {eqnarray*}
&&\hskip-.5cm\langle A^\kappa_{\rm hom}z,z\rangle\\
&&=\min\Bigl\{\int_{\mathbb R^d\times (0,1)^d} \frac{1}{\kappa^d}\varrho\Bigl(\frac{|x-y|}{\kappa}\Bigr) a(y)\frac{|\langle z,y-x\rangle+\varphi(y)-\varphi(x)|^2}{|x-y|^2}dxdy: \varphi\  1\hbox{-periodic}\Bigr\}.
\end{eqnarray*}
In order to prove this result, localization techniques for limits of non-local functionals must be used (see \cite{AABPT} and also \cite{BP}). Comparing  functionals $F_{\e,s}$ with those in \eqref{functional-1}, a main additional difficulty is due to the different nature of the dependence on $\delta$ and $s$, respectively, so that the case $1-s\sim \e^2$ cannot be directly set as a homogenization-concentration problem, contrary to the case $\delta\sim\e$.

\section{Statement of the result and preliminaries}
Let $a$ be a $1$-periodic continuous function such that 
\begin{equation}\label{eq:bounds}
0<\alpha \leq a(x) \leq \beta<+\infty
\end{equation}
for every $x\in\R^d$. For $\e,s\in(0,1)$ we introduce the functional 
\begin{equation}\label{functional}
    F_{\varepsilon,s}(u)=(1-s)\iint_{\Omega\times \Omega} a\left(\frac{x}{\varepsilon}\right)\frac{\lvert u(x)-u(y)\rvert^2}{\lvert x-y\rvert^{d+2s}}dxdy
\end{equation}
 for $u\in W^{s,2}(\Omega)$.

We assume that the $\e$ is a positive parameter and that $s$ is a function of $\e$ valued in $(0,1)$ approaching $1$ as $\e\to0$. Since $s=s(\e)$, as a shorthand, $F_{\e, s}$ will be denoted by $F_\e$.
We prove that if $1-s \ll\e^2$, then the separation of scales described in the Introduction holds, as if we were passing to the limit first letting $s\to1$ and then letting $\e\to 0$. 

\begin{theorem}
Let $\Omega$ be a bounded open set with Lipschitz boundary, $\e\in(0,1), \, s=s_\e$, and let $F_{\varepsilon}$ be the functional defined in \eqref{functional}. If \begin{equation}\label{eq:1-se}
\lim_{\varepsilon \to 0}\frac{1-s}{\varepsilon^2}=0,
\end{equation}
then  
    \begin{equation}\label{mainthm}
    \Gamma\hbox{-}\lim_{\e\to0} F_{\varepsilon}(u) = \frac{\sigma_{d-1}}{2d}\int_\Omega \langle A_{\rm hom}\nabla u,\nabla u\rangle\,dx=:F_{\rm hom}(u)
    \end{equation} 
for every $u\in W^{1,2}(\Omega)$, where the $\Gamma$-limit is computed with respect to the $L^2(\Omega)$ convergence, $A_{\rm hom}$ is given by \eqref{homfor}, and $\sigma_{d-1}:={\mathcal H}^{d-1}(S^{d-1})$ is the $(d-1)$-dimensional Hausdorff measure of the unit sphere in $\mathbb{R}^d$.
\end{theorem}

We note that the hypothesis that $\Omega$ has a Lipschtiz boundary is exploited only in proving, by means of  a density argument, the $\Gamma$-$\limsup$ inequality.

\begin{remark}\rm Theorem 1 can be generalized to functionals modelled on $(s,p)$-Gagliardo seminorms with $p>1$, of the form 
\begin{equation}\label{functional-p}
    F_{\varepsilon,s}(u)=(1-s)\iint_{\Omega\times \Omega} a\left(\frac{x}{\varepsilon}\right)\frac{\lvert u(x)-u(y)\rvert^p}{\lvert x-y\rvert^{d+ps}}dxdy
\end{equation}
 for $u\in W^{s,p}(\Omega)$. The subcritical regime is then $1-s<\!<\e^p$, and the resulting $\Gamma$-limit is 
$$C_{d,p}\int_\Omega f_{\rm hom} (\nabla u)\,dx,$$
defined on $W^{1,p}(\Omega)$, with $C_{d,p}$ the constant appearing in the corresponding result in \cite{bbm}, and $f_{\rm hom}$ defined by 
\begin{equation}\label{homfor-p}
f_{\rm hom}(z)=\min\Bigl\{\int_{(0,1)^d} a(y)|z+\nabla \varphi(y)|^pdy: \varphi\  1\hbox{-periodic}\Bigr\}.
\end{equation}
This result can be obtained with few changes and some heavier notation from the case $p=2$. 
\end{remark}

We preliminarily state a lemma that allows us to take into account only those interactions due to pairs of sufficiently close points in $\Omega$.

\begin{lemma}\label{lemma:locality} Let $(u_\e)_\e\subseteq L^2(\Omega)$ be bounded and let $(r_\e)_\e$ be a sequence of positive real numbers such that
\be\label{eq:lemmare}
\lim_{\e\to0}\frac{1-s_\e}{r_\e^2}=0.
\ee
Then 
\[
\lim_{\e\to0} \biggl(F_{\e}(u_\e)-(1-s_\e) \iint_{\Omega\times \Omega \cap \{(x,y) : |x-y|\leq r_\e\}}a\left(\frac{x}{\e}\right)\frac{|u_\e(x)-u_\e(y)|^2}{|x-y|^{d+2s_\e}}\,dxdy\biggr)=0 .
\]
\end{lemma}

\begin{proof}
 By the convexity of $x\mapsto x^2$ and by \eqref{eq:bounds}, we have
\begin{eqnarray} \label{inlemma} && \nonumber
\hskip-2cm(1-s_\e)\iint_{\Omega\times \Omega \cap \{(x,y) : |x-y| > r_\e\}}a\left(\frac{x}{\e}\right)\frac{|u_\e(x)-u_\e(y)|^2}{|x-y|^{d+2s_\e}}\,dxdy  \\ \nonumber
&\leq& (1-s_\e)\, 4 \beta \iint_{\Omega\times \Omega \cap \{(x,y) : |x-y| > r_\e\}}\frac{|u_\e(x)|^2}{|x-y|^{d+2s_\e}}\,dxdy    \\ 
&=& (1-s_\e)\, 4 \beta \int_{\R^d\setminus B_{r_\e}(0)}\frac{1}{|\xi|^{d+2s_\e}}\,d\xi\int_\Omega|u_\e(x)|^2\,dx,
\end{eqnarray}
where we also took advantage of the symmetry of the integrand to get the inequality.

Since $(u_\e)_\e$ is bounded in $L^2(\Omega)$, \eqref{inlemma} is estimated from above (up to a constant factor) by
\begin{eqnarray*}
    (1-s_\e)\int_{\R^d\setminus B_{r_\e}(0)}\frac{1}{|\xi|^{d+2s_\e}}\,d\xi& = &(1-s_\e)\mathcal{H}^{d-1}(S^{d-1})\int_{r_\e}^{\infty}\rho^{-1-2s_\e}d\rho\\
    &=& \mathcal{H}^{d-1}(S^{d-1})(1-s_\e)\frac{r_\e^{-2s_\e}}{2s_\e}\\
    &\leq& \mathcal{H}^{d-1}(S^{d-1})(1-s_\e)\frac{r_\e^{-2}}{2s_\e},
\end{eqnarray*}
which tends to $0$ by \eqref{eq:lemmare}, completing our proof.
\end{proof}

\begin{remark}\label{re:remark} We will apply the previous Lemma setting
\[
r_\e= \e \qquad \hbox{ or } \qquad r_\e=\sqrt{\e\sqrt{1-s_\e}}
\]
in accordance with our convenience.

Note that with the second choice it holds
$$\sqrt{1-s_\e}<\!<r_\e<\!<\e,$$
so that, in particular,
\begin{equation}\label{eq:ratio}
\lim_{\e\to0}\frac{r_\e}{\e}=0.
    \end{equation}
It is useful to observe that, in both circumstances, we have
\[
\lim_{\e\to0}|\log r_\e^{1-s_\e}| = \lim_{\e\to0}(1-s_\e)|\log r_\e|=0
\]
which implies
\be\label{eq:power}
\lim_{\e\to0} r_\e^{1-s_\e}=1.
\ee
\end{remark}
In the proof of the liminf inequality we take advantage of a discretization argument based on the construction of proper lattices, which can be parameterized on orthogonal bases.

\begin{remark}[Notation for the set of orthogonal bases]\label{notation}
Following the notation of \cite{solci2023} we define the set of orthonormal bases of $\mathbb{R}^d$
\[
V:=\{\overline{\nu}=(\nu_1,...,\nu_d) : \nu_j \in S^{d-1} \text{ such that } \langle\nu_i, \nu_j\rangle=0 \text{ for } i\neq j \}
\]
and observe that $V$ has Hausdorff dimension equal to $k_d:=d(d-1)/2$. 
For every $n\in\{1,...,d\}$ and fixed $\nu\in S^{d-1}$ we define
\[
V_n^{\nu}:= \{\overline{\nu}\in V \hbox{ such that } \nu_n=\nu\},
\]
whose Hausdorff dimension is $k_d-(d-1)$. Note that we have
\begin{equation} \label{hausdorffmeas}
\mathcal{H}^{k_d-(d-1)}(V_n^\nu)=\frac{\mathcal{H}^{k_d}(V)}{\mathcal{H}^{d-1}(S^{d-1})},
\end{equation}
and that in general, the formula
\begin{equation}\label{disintegration}
\int_{V}f(\overline{\nu})\, d\mathcal{H}^{k_d}(\overline{\nu}) = \int_{S^{d-1}} 
\int_{V_n^\nu} f^\nu(\overline{\nu}_n) \,d\mathcal{H}^{k_d-(d-1)} (\overline{\nu}_n)\,d\mathcal{H}^{d-1}(\nu),
\end{equation}
holds with $\overline{\nu}_n:=(\nu_1,...,\nu_{n-1}, \nu_{n+1},...,\nu_d)$ and $f^\nu(\overline{\nu}_n):= f(\overline{\nu})$ for every $n\in\{1,...,d\}$ and $f$ non negative measurable function.

\smallskip

Given $\rho>0$ and $\overline{\nu}\in V$, we define $\mathbb Z^d_{\rho\overline{\nu}}:=\{z_1\rho\nu_1+z_2\rho\nu_2+...+z_d\rho\nu_d : (z_1,...,z_d)\in \mathbb Z^d\}$ and $Q_{\rho\overline{\nu}}$ as the cube described by the orthogonal basis $\{\rho\nu_1,...,\rho\nu_d\}$.
\end{remark}

\begin{remark}[Kuhn's decomposition]\label{re:remark1}  
A cube $Q_{\rho\overline{\nu}}$ defined as in the previous remark can be further decomposed into $d!$ $d$-simplices of size $\rho^d/d!$ through {\em Kuhn's decomposition} in the following way: for every $\tau$ permutation of $d$ indices, we define $\Delta_{\rho\overline{\nu}}^{\tau}$ as the simplex described by the vertices $\rho\nu_{\tau(1)}, \rho\nu_{\tau(1)}+\rho\nu_{\tau(2)},..., \rho\nu_{\tau(1)}+\rho\nu_{\tau(2)}+...+\rho\nu_{\tau(d)}$. As shown in \cite[Lemma 1]{Kuhn1960}, as $\tau$ vary among all the permutations, we get a family of $d!$ simplices which constitutes the desired partition.

We denote the vertices of each simplex by 
\[
\Delta_{\rho\overline{\nu}}^{\tau,0}=0, \quad    \Delta_{\rho\overline{\nu}}^{\tau,j}:=\rho\nu_{\tau(1)}+\rho\nu_{\tau(2)}+...+\rho\nu_{\tau(j)}
\quad \hbox{ for } j=1,...,d
\]
and observe that, with this choice, it holds
\begin{equation} \label{vertices}
\Delta_{\rho\overline{\nu}}^{\tau, j}-\Delta_{\rho\overline{\nu}}^{\tau, j-1}=\rho\nu_{\tau(j)} \qquad \hbox{ for } j=1,...,d.
\end{equation}
\end{remark}

\section{Proof of the result}

\subsection{\bf Liminf inequality}
Throughout this section we write $s:=s_\e$, $r:=r_\e$ and $\sigma_{d-1}:=\mathcal{H}^{d-1}(S^{d-1})$ in order to simplify the notation.

Let $u\in L^{2}(\Omega)$ and consider a sequence $(u_\e)_\e$ converging to $u$ in $L^2(\Omega)$. Without loss of generality, we may assume that $\sup_\e F_{\e}(u_\e)< \infty$ which implies that $u_\e\in W^{s_,2}(\Omega)$ for every $\e>0$. 

In light of Lemma \ref{lemma:locality}, we aim at proving that $u\in 
W^{1,2}(\Omega')$ for every $\Omega'$ open subset well contained in $\Omega$ and that
\begin{eqnarray*}\label{mainclaim}
   &&\hskip-2cm\liminf_{\e\to0} (1-s) \iint_{\Omega\times \Omega \cap \{(x,y) : |x-y|\leq r\}}a\left(\frac{x}{\e}\right) \frac{|u_\e(x)-u_\e(y)|^2}{|x-y|^{d+2s}}\,dxdy \\
   &\geq& \frac{\sigma_{d-1}}{2d}\int_{\Omega'} \langle A_{\rm hom}\nabla  u,\nabla u\rangle\,dx,
\end{eqnarray*}
 where $r=\e^{\frac{1}{2}}(1-s)^{\frac{1}{4}}$. By the independence of the left-hand side from $\Omega'$ in particular this implies that $u\in W^{1,2}(\Omega)$.

Applying  Lemma \ref{lemma:locality}, the change of variables $\eta:=y-x$ and $\xi:= \eta/r$, and then the coarea formula, we get
\begin{eqnarray*} 
    \nonumber&&
\hskip-1.5cm    F_\e(u_\e)+o_\e(1) \\ \nonumber
&=& (1-s) \iint_{\Omega\times \Omega \cap \{(x,y) : |x-y|\leq r\}}a\left(\frac{x}{\e}\right)\frac{|u_\e(x)-u_\e(y)|^2}{|x-y|^{d+2s}}\,dxdy \\ \nonumber
    &=& (1-s) \int_{B_{r}(0)}\frac{1}{|\eta|^{d-2(1-s)}}\int_{\{x\in\Omega \,:\,x+\eta\in\Omega\}}a\left(\frac{x}{\e}\right)\frac{|u_\e(x+\eta)-u_\e(x)|^2}{|\eta|^2}\,dxd\eta \\ \nonumber
    &=& (1-s) \int_{B_1(0)}\frac{r^{2(1-s)}}{|\xi|^{d-2(1-s)}}\int_{\{x\in\Omega \,:\,x+r\xi\in\Omega\}}a\left(\frac{x}{\e}\right)\frac{|u_\e(x+r\xi)-u_\e(x)|^2}{|r\xi|^2}\,dxd\xi \\ \nonumber
    &=& (1-s) \int_0^1 \rho^{d-1} \int_{S^{d-1}} \frac{r^{2(1-s)}}{\rho^{d-2(1-s)}}\\ \nonumber 
    && \qquad \qquad \qquad \int_{\{x\in\Omega \,:\,x+r\rho\nu\in\Omega\}}a\left(\frac{x}{\e}\right)\frac{|u_\e(x+r\rho\nu)-u_\e(x)|^2}{|r\rho|^2} dxd\mathcal{H}^{d-1}(\nu) d\rho.
\end{eqnarray*}
For $\overline{\nu}\in V$ (recall the notation of Remark \ref{notation}) we set
\[
\mathcal{I}^r_{\rho\overline{\nu}}:=\{k\in \mathbb Z_{\rho\overline{\nu}}^d : rk+rQ_{\rho\overline{\nu}} \subset \subset \Omega\},
\]
and note that for every $\nu\in S^{d-1}$ it holds
\[
\bigcup_{k\in \mathcal{I}^r_{\rho \overline{\nu}}} rk+rQ_{\rho\overline{\nu}} \subseteq \{x\in \Omega : x+r\rho\nu\in\Omega\}
\]
 for every $\overline{\nu}\in V$. Hence, we have
\begin{eqnarray}\nonumber \label{liminf1}  
&&\hskip-1.5cm F_\e(u_\e)+o_\e(1) \\
&& \hskip-1.5cm \geq(1-s)  \int_0^1 \frac{r^{2(1-s)}}{\rho^{-1+2s}} \int_{S^{d-1}} \sum_{k\in \mathcal{I}^r_{\rho\overline{\nu}}} \int_{rk+rQ_{\rho\overline{\nu}}} {\hskip-0.3cm} a\left(\frac{x}{\e}\right)\frac{|u_\e(x+r\rho\nu)-u_\e(x)|^2}{|r\rho|^2} dxd\mathcal{H}^{d-1}(\nu) d\rho.
 \end{eqnarray}
 
We shall exploit the uniform continuity of the function $a$ to factor it out from the inner integral; for this reason, we introduce a modulus of continuity $\omega:[0,+\infty)\rightarrow[0,+\infty)$; that is, an increasing continuous function such that $\omega(0)=0$ and $|a(x_1)-a(x_2)|\leq\omega(|x_1-x_2|)$ for every $x_1,x_2\in\Rd$.

We rewrite the right-hand side of \eqref{liminf1} as
\begin{eqnarray*}\nonumber &&
    \hskip-0.8cm (1-s) \int_0^1 \frac{r^{2(1-s)}}{\rho^{-1+2s}} \int_{S^{d-1}} \sum_{k\in \mathcal{I}^r_{\rho\overline{\nu}}}  a\left(\frac{rk}{\e}\right) \int_{rk+rQ_{\rho\overline{\nu}}}\frac{|u_\e(x+r\rho\nu)-u_\e(x)|^2}{|r\rho|^2} dxd\mathcal{H}^{d-1}(\nu) d\rho \\  \nonumber 
    && \hskip-1.1cm + (1-s) \int_0^1 \frac{r^{2(1-s)}}{\rho^{-1+2s}}  \\
    &&\int_{S^{d-1}}\sum_{k\in \mathcal{I}^r_{\rho\overline{\nu}}} \int_{rk+rQ_{\rho\overline{\nu}}} {\hskip-0.1cm} \Bigl(a\left(\frac{x}{\e}\right)-a\Bigl(\frac{rk}{\e}\Bigr)\Bigr)\frac{|u_\e(x+r\rho\nu)-u_\e(x)|^2}{|r\rho|^2} dxd\mathcal{H}^{d-1}(\nu) d\rho, 
\end{eqnarray*}
and we note that the last term is negligible as $\e\to0$ as its absolute value is estimated from above by
\begin{eqnarray*}  && \hskip-2cm \omega\Bigl(\frac{r\sqrt{d}}{\e}\Bigr) (1-s) \int_0^1 \frac{r^{2(1-s)}}{\rho^{-1+2s}} \int_{S^{d-1}} \int_{\{x+r\rho\nu \in \Omega\}} \frac{|u_\e(x+r\rho\nu)-u_\e(x)|^2}{|r\rho|^2} dxd\mathcal{H}^{d-1}(\nu) d\rho \\
&=& \omega\Bigl(\frac{r\sqrt{d}}{\e}\Bigr) (1-s) \int_{B_1(0)}\frac{r^{2(1-s)}}{|\xi|^{d-2(1-s)}}\int_{\{x+r\xi \in \Omega\}}\frac{|u_\e(x+r\xi)-u_\e(x)|^2}{|r\xi|^2}\,dx d\xi \\ \nonumber
   &\leq& \omega\Bigl(\frac{r\sqrt{d}}{\e}\Bigr)\frac{1}{\alpha} F_\e(u_\e),
\end{eqnarray*}
which tends to $0$ since by the assumption $\sup_\e F_\e(u_\e) <+\infty$ and since $r/\e\to0$ by \eqref{eq:ratio}.

By this fact and by formulas \eqref{hausdorffmeas} and \eqref{disintegration}, inequality \eqref{liminf1} turns into
\begin{eqnarray} && \nonumber
\hskip-0.8cm F_\e(u_\e)+o_\e(1) \\ \nonumber
{\hskip-1.5cm} &\geq&(1-s) \int_0^1 \frac{r^{2(1-s)}}{\rho^{-1+2s}} \int_{S^{d-1}} \sum_{k\in \mathcal{I}^r_{\rho\overline{\nu}}} a\left(\frac{rk}{\e}\right)\int_{rk+rQ_{\rho\overline{\nu}}} \frac{|u_\e(x+r\rho\nu)-u_\e(x)|^2}{|r\rho|^2} dxd\mathcal{H}^{d-1}(\nu) d\rho\\ \nonumber
\hskip-1.5cm&=& (1-s) \int_0^1  \frac{r^{2(1-s)}}{\rho^{-1+2s}} \int_{S^{d-1}} \frac{1}{d} \sum_{n=1}^d \frac{\sigma_{d-1}}{\mathcal{H}^{k_d}(V)} \\ \nonumber
&& \qquad \int_{V_n^\nu} \sum_{k\in \mathcal{I}^r_{\rho \overline{\nu}}} a\left(\frac{rk}{\e}\right)\int_{rk+rQ_{\rho\overline{\nu}}}\frac{|u_\e(x+r\rho\nu)-u_\e(x)|^2}{|r\rho|^2} dx d\mathcal{H}^{k_d-(d-1)}(\overline{\nu}_n) d\mathcal{H}^{d-1}(\nu) d\rho \\ \nonumber
&=& \frac{1-s}{d} \frac{\sigma_{d-1}}{\mathcal{H}^{k_d}(V)} \int_0^1  \frac{r^{2(1-s)}}{\rho^{-1+2s}}  \\ \label{lastterm}
&& \qquad \qquad  \int_V  \sum_{n=1}^d \sum_{k\in \mathcal{I}^r_{\rho \overline{\nu}}}  a\left(\frac{rk}{\e}\right) \int_{rk+rQ_{\rho\overline{\nu}}}\frac{|u_\e(x+r\rho\nu_n)-u_\e(x)|^2}{|r\rho|^2} dx d\mathcal{H}^{k_d}(\overline{\nu}) d\rho.
\end{eqnarray}
To produce a lower bound for \eqref{lastterm}, we take advantage of a discretization argument based on piecewise-affine  auxiliary functions.

Given sufficiently small $\e>0$, $\rho\in (0,1)$ and $\overline{\nu}\in V$ we define the function $u_\e^{\rho\overline{\nu}}$ in two steps:

    $\bullet$ first, we assign values on the lattice $\mathcal{I}^r_{\rho\overline{\nu}}$ setting
    \[
    u_\e^{\rho\overline{\nu}}(rk)= \displaystyle \frac{1}{|r\rho|^d}\int_{rk+rQ_{\rho\overline{\nu}}} u_\e\,dx \qquad \hbox{ for every } k\in \mathcal{I}^r_{\rho\overline{\nu}};
    \] 
    
    $\bullet$ then, given $k$ in the coarser lattice $\mathcal{I}^{2r}_{\rho\overline{\nu}}$, consider the cube $rk+rQ_{\rho\overline{\nu}}$, $\tau$ 
a permutation of the indices $\{1,...,d\}$, and $rk+r\Delta^\tau_{\rho\overline{\nu}}$ the corresponding simplex in Kuhn's decomposition, on such simplex we define $u^{\rho\overline{\nu}}_\e$ being the affine interpolation of the previously defined values $u^{\rho \overline{\nu}}_\e(rk) = u^{\rho \overline{\nu}}_\e(rk+r\Delta^{\tau,0}_{\rho\overline{\nu}}),\, u^{\rho\overline{\nu}}_\e(rk+r\Delta^{\tau, 1}_{\rho\overline{\nu}}),..., \,u^{\rho\overline{\nu}}_\e(rk+r\Delta^{\tau, d}_{\rho\overline{\nu}})$.

Note that $u_\e^{\rho\overline{\nu}}$ is well defined as a continuous piecewise-affine function on $\bigcup_{k\in \mathcal{I}^{2r}_{\rho\overline{\nu}}} rk+rQ_{\rho\overline{\nu}}$; hence, on each simplex $rk+r\Delta_{\rho\overline{\nu}}^\tau$, its gradient is constant and by \eqref{vertices} it holds
\begin{eqnarray*} 
\nonumber
    &&\hskip-1.5cm\int_{rk+r\Delta_{\rho\overline{\nu}}^\tau} |\nabla u_\e^{\rho\overline{\nu}}|^2\,dx \\
    &=& \sum_{j=1}^d \int_{rk+r\Delta_{\rho\overline{\nu}}^\tau} \frac{|u_\e^{\rho\overline{\nu}}(rk+r\Delta_{\rho\overline{\nu}}^{\tau,j})-u_\e^{\rho\overline{\nu}}(rk+r\Delta_{\rho\overline{\nu}}^{\tau,j-1})|^2}{|r\rho|^2}\,dx      \\ \nonumber
    &=& \frac{1}{|r\rho|^2}\sum_{j=1}^d \int_{rk+r\Delta_{\rho\overline{\nu}}^\tau} \biggl|\frac{1}{|r\rho|^d}\Bigl(\int_{rk+r\Delta_{\rho\overline{\nu}}^{\tau,j}+rQ_{\rho\overline{\nu}}}\hspace{-0.2cm}u_\e\, dy-\int_{rk+r\Delta_{\rho\overline{\nu}}^{\tau,j-1}+rQ_{\rho\overline{\nu}}}\hspace{-0.2cm} u_\e\, dy \Bigr)\biggr|^2dx \\ \nonumber
    &=& \frac{1}{|r\rho|^2}\sum_{j=1}^d \int_{rk+r\Delta_{\rho\overline{\nu}}^\tau} \biggl| \frac{1}{|r\rho|^d}\Bigl(\int_{rk+r\Delta_{\rho\overline{\nu}}^{\tau,j-1}+rQ_{\rho\overline{\nu}}} \hspace{-0.2cm} u_\e(y+r\rho\nu_{\tau(j)})-u_\e(y)\, dy\Bigr)\biggr|^2 dx \\ \nonumber
    &\leq& \frac{1}{|r\rho|^2}\sum_{j=1}^d \int_{rk+r\Delta_{\rho\overline{\nu}}^\tau} \frac{1}{|r\rho|^d}\int_{rk+r\Delta_{\rho\overline{\nu}}^{\tau,j-1}+rQ_{\rho\overline{\nu}}} \hspace{-0.2cm} |u_\e(y+r\rho\nu_{\tau(j)})-u_\e(y)|^2 dy\,dx \\ \nonumber  
    &=& \frac{1}{|r\rho|^d}\sum_{j=1}^d \int_{rk+r\Delta_{\rho\overline{\nu}}^\tau} \int_{rk+r\Delta_{\rho\overline{\nu}}^{\tau,j-1}+rQ_{\rho\overline{\nu}}} \frac{|u_\e(y+r\rho\nu_{\tau(j)})-u_\e(y)|^2}{|r\rho|^2}\, dy\,dx \\ 
    &=& \frac{1}{d!} \sum_{j=1}^d \int_{rk+r\Delta_{\rho\overline{\nu}}^{\tau,j-1}+rQ_{\rho\overline{\nu}}} \frac{|u_\e(y+r\rho\nu_{\tau(j)})-u_\e(y)|^2}{|r\rho|^2}\, dy,
\end{eqnarray*}
where we used Jensen's inequality and the fact that $|rk+r\Delta_{\rho\overline{\nu}}^\tau|=|r\rho|^d/d!$.

Summing over all permutations $\tau$ and $k\in \mathcal{I}^{2r}_{\rho\overline{\nu}}$ and reinstating the coefficient $a$, we get
\begin{eqnarray} \label{liminf3} && \nonumber 
\hskip-1.8cm \sum_{k\in \mathcal{I}^{2r}_{{\rho\overline{\nu}}}} a\left(\frac{rk}{\e}\right) \int_{rk+rQ_{\rho\overline{\nu}}} |\nabla u_\e^{\rho\overline{\nu}}|^2\,dx   
\\ \nonumber
&=& \sum_{k\in \mathcal{I}^{2r}_{{\rho\overline{\nu}}}} a\left(\frac{rk}{\e}\right)\sum_\tau \int_{rk+r\Delta_{\rho\overline{\nu}}^\tau} |\nabla u_\e^{\rho\overline{\nu}}|^2\,dx \\ \nonumber
    &\leq& \sum_{k\in \mathcal{I}^{2r}_{{\rho\overline{\nu}}}} a\left(\frac{rk}{\e}\right)\sum_\tau \frac{1}{d!} \sum_{j=1}^d \int_{rk+r\Delta_{\rho\overline{\nu}}^{\tau,j-1}+rQ_{\rho\overline{\nu}}} \frac{|u_\e(y+r\rho\nu_{\tau(j)})-u_\e(y)|^2}{|r\rho|^2}\, dy \\ \label{term}
    &=& \frac{1}{d!} \sum_{n=1}^d \sum_{\{\tau,\,j \,| \, \tau(j)=n\}} \sum_{k\in \mathcal{I}^{2r}_{{\rho\overline{\nu}}}} a\left(\frac{rk}{\e}\right) \int_{rk+r\Delta_{\rho\overline{\nu}}^{\tau,j-1}+rQ_{\rho\overline{\nu}}} \hspace{-0.4 cm} \frac{|u_\e(y+r\rho\nu_n)-u_\e(y)|^2}{|r\rho|^2}\, dy. 
\end{eqnarray}
Keeping $n,j,\tau$ fixed, we put $h:=k+\Delta_{\rho\overline{\nu}}^{\tau,j-1}$ and note that $k\in \mathcal{I}^{2r}_{\rho\overline{\nu}}$ implies $h\in\mathcal{I}^{r}_{\rho\overline{\nu}}$  so that \eqref{term} is less than or equal to
\begin{eqnarray*} && \nonumber
    \hskip-2.5cm \frac{1}{d!} \sum_{n=1}^d \sum_{\{\tau,\,j \,| \, \tau(j)=n\}}\sum_{h\in \mathcal{I}^r_{{\rho\overline{\nu}}}} a\left(\frac{rh}{\e}\right)  \int_{rh+rQ_{\rho\overline{\nu}}} \hspace{-0.4 cm} \frac{|u_\e(y+r\rho\nu_n)-u_\e(y)|^2}{|r\rho|^2}\, dy + E(\overline{\nu}) \\ \label{term1}
    &=& \sum_{n=1}^d \sum_{h\in \mathcal{I}^r_{{\rho\overline{\nu}}}} a\left(\frac{rh}{\e}\right) \int_{rh+rQ_{\rho\overline{\nu}}} \frac{|u_\e(y+r\rho\nu_n)-u_\e(y)|^2}{|r\rho|^2}\, dy + E(\overline{\nu}),
\end{eqnarray*}
where we used that the cardinality of the set  $\{\tau, j\, |\, \tau(j)=n\}$ is $d!$ for every $n=1,...,d$, and we set
\begin{equation*}
E(\overline{\nu}):= \sum_{n=1}^d \sum_{h\in \mathcal{I}^{r}_{{\rho\overline{\nu}}}} \biggl(a\biggl(\frac{rh-r\Delta_{\rho\overline{\nu}}^{\tau,j-1}}{\e}\biggr)-a\left(\frac{rh}{\e}\right)\biggr) \int_{rh+rQ_{\rho\overline{\nu}}} \frac{|u_\e(y+r\rho\nu_n)-u_\e(y)|^2}{|r\rho|^2}\, dy.
\end{equation*}

Summarizing, we have proved that for every $\e, \rho, \overline{\nu}$, it holds
\begin{eqnarray}  \label{ineq}\nonumber
&&\hskip-2cm\sum_{n=1}^d \sum_{k\in \mathcal{I}^{r}_{{\rho\overline{\nu}}}} a\left(\frac{rk}{\e}\right) \int_{rk+rQ_{\rho\overline{\nu}}} \hskip-0.3cm\frac{|u_\e(y+r\rho\nu_n)-u_\e(y)|^2}{|r\rho|^2}\,dy \\
&&\geq \sum_{k\in \mathcal{I}^{2r}_{{\rho\overline{\nu}}}} a\left(\frac{rk}{\e}\right) \int_{rk+rQ_{\rho\overline{\nu}}} |\nabla u_\e^{\rho\overline{\nu}}|^2\,dx  
-\,\, E(\overline{\nu}),
\end{eqnarray}
and therefore, having achieved an estimate in terms of the inner intergal in \eqref{lastterm}, we get
\begin{eqnarray} && \nonumber \hskip-1.2cm F_\e(u_\e)+o_\e(1) \\ \nonumber
&\geq& \frac{1-s}{d} \frac{\sigma_{d-1}}{\mathcal{H}^{k_d}(V)} \int_0^1  \frac{r^{2(1-s)}}{\rho^{-1+2s}} \\ \label{midterm0}
&& \quad \quad \int_V  \sum_{n=1}^d \sum_{k\in \mathcal{I}^r_{\rho \overline{\nu}}}  a\left(\frac{rk}{\e}\right) \int_{rk+rQ_{\rho\overline{\nu}}}\frac{|u_\e(x+r\rho\nu_n)-u_\e(x)|^2}{|r\rho|^2} dx d\mathcal{H}^{k_d}(\overline{\nu}) d\rho \\ \label{midterm}
&\geq& \frac{1-s}{d} \frac{\sigma_{d-1}}{\mathcal{H}^{k_d}(V)} \int_0^1  \frac{r^{2(1-s)}}{\rho^{-1+2s}}\int_V \sum_{k\in \mathcal{I}^{2r}_{{\rho\overline{\nu}}}} a\left(\frac{rk}{\e}\right) \int_{rk+rQ_{\rho\overline{\nu}}} |\nabla u_\e^{\rho\overline{\nu}}|^2\,dx   d\mathcal{H}^{k_d}(\overline{\nu}) d\rho \\ \label{midterm2}
&& \qquad \qquad \qquad \qquad \qquad - \, \frac{1-s}{d} \frac{\sigma_{d-1}}{\mathcal{H}^{k_d}(V)} \int_0^1  \frac{r^{2(1-s)}}{\rho^{-1+2s}}\int_V E(\overline{\nu})  d\mathcal{H}^{k_d}(\overline{\nu}) d\rho. 
\end{eqnarray}

Note that \eqref{midterm2} vanishes as $\e\to0$; indeed, considering the modulus of continuity, the bound from above \eqref{eq:bounds} and the bound from above involving \eqref{midterm0}, we have
\begin{eqnarray*} && \nonumber
\hskip-1cm  \frac{1-s}{d} \frac{\sigma_{d-1}}{\mathcal{H}^{k_d}(V)} \int_0^1  \frac{r^{2(1-s)}}{\rho^{-1+2s}} \int_V E(\overline{\nu})\,  d\mathcal{H}^{k_d}(\overline{\nu}) d\rho \\ \nonumber
    &\leq& \omega\Bigl(\frac{r\sqrt{d}}{\e}\Bigr)\frac{1-s}{d} \frac{\sigma_{d-1}}{\mathcal{H}^{k_d}(V)}\int_0^1  \frac{r^{2(1-s)}}{\rho^{-1+2s}} \\ \nonumber
    && \qquad \qquad \qquad \qquad \int_V \sum_{n=1}^d \sum_{h\in \mathcal{I}^{r}_{{\rho\overline{\nu}}}}  \int_{rh+rQ_{\rho\overline{\nu}}} \frac{|u_\e(y+r\rho\nu_n)-u_\e(y)|^2}{|r\rho|^2}\, dy d\mathcal{H}^{k_d}(\overline{\nu})d\rho \\ \nonumber
    &\leq& \omega\Bigl(\frac{r\sqrt{d}}{\e}\Bigr)\frac{1-s}{d} \frac{\sigma_{d-1}}{\mathcal{H}^{k_d}(V)}\frac{1}{\alpha}\int_0^1  \frac{r^{2(1-s)}}{\rho^{-1+2s}} \\ \nonumber
    && \qquad \quad \qquad \int_V \sum_{n=1}^d \sum_{h\in \mathcal{I}^{r}_{{\rho\overline{\nu}}}}  a\left(\frac{rh}{\e}\right) \int_{rh+rQ_{\rho\overline{\nu}}}\frac{|u_\e(y+r\rho\nu_n)-u_\e(y)|^2}{|r\rho|^2}\, dy d\mathcal{H}^{k_d}(\overline{\nu})d\rho \\ \nonumber
    &\leq& \omega\Bigl(\frac{r\sqrt{d}}{\e}\Bigr)\frac{1}{\alpha}F_\e(u_\e)+o_\e(1),
\end{eqnarray*}
which tends to $0$ since $\sup_\e F_\e(u_\e)<+\infty$ and since $r/\e\to0$ by \eqref{eq:ratio}.

As for \eqref{midterm}, we can reinsert the coefficient $a$ in the integral; in particular, taking into account the modulus of continuity and applying \eqref{eq:bounds}, \eqref{ineq}, and once more the inequality leading to \eqref{midterm0}, we have
\begin{eqnarray*} && \nonumber
    \hskip-1cm \frac{1-s}{d} \frac{\sigma_{d-1}}{\mathcal{H}^{k_d}(V)} \biggl| \int_0^1  \frac{r^{2(1-s)}}{\rho^{-1+2s}}\int_V \sum_{k\in \mathcal{I}^{2r}_{{\rho\overline{\nu}}}} \int_{rk+rQ_{\rho\overline{\nu}}} \biggl(a\biggl(\frac{rk}{\e}\biggr)-a\biggl(\frac{x}{\e}\biggr)\biggr)  |\nabla u_\e^{\rho\overline{\nu}}|^2\,dx   d\mathcal{H}^{k_d}(\overline{\nu}) d\rho \biggr| \\ \nonumber
    &\leq&  \omega\Bigl(\frac{r\sqrt{d}}{\e}\Bigr)\frac{1-s}{d} \frac{\sigma_{d-1}}{\mathcal{H}^{k_d}(V)} \int_0^1  \frac{r^{2(1-s)}}{\rho^{-1+2s}}\int_V \sum_{k\in \mathcal{I}^{2r}_{{\rho\overline{\nu}}}} \int_{rk+rQ_{\rho\overline{\nu}}}  |\nabla u_\e^{\rho\overline{\nu}}|^2\,dx   d\mathcal{H}^{k_d}(\overline{\nu}) d\rho \\ \nonumber
    &\leq&  \omega\Bigl(\frac{r\sqrt{d}}{\e}\Bigr)\frac{1-s}{d} \frac{\sigma_{d-1}}{\mathcal{H}^{k_d}(V)}\frac{1}{\alpha} \int_0^1  \frac{r^{2(1-s)}}{\rho^{-1+2s}} \\ \nonumber
    && \qquad  \int_V \sum_{n=1}^d \sum_{k\in \mathcal{I}^{r}_{{\rho\overline{\nu}}}} a\left(\frac{rk}{\e}\right) \int_{rk+rQ_{\rho\overline{\nu}}} \hskip-0.3cm\frac{|u(y+r\rho \nu_n)-u(y)|^2}{|r\rho|^2}\,dy + E(\overline{\nu}) d\mathcal{H}^{k_d}(\overline{\nu}) d\rho \\ \nonumber
    &\leq& \omega\Bigl(\frac{r\sqrt{d}}{\e}\Bigr)\frac{1}{\alpha} F_\e(u_\e) + o_\e(1),
\end{eqnarray*}
and we conclude as above.

We finally obtain
\begin{eqnarray} && \nonumber
 \hskip-1.7cm   F_\e(u_\e)+o_\e(1) \\ \nonumber
&\geq& \frac{1-s}{d} \frac{\sigma_{d-1}}{\mathcal{H}^{k_d}(V)} \int_0^1  \frac{r^{2(1-s)}}{\rho^{-1+2s}}\int_V \sum_{k\in \mathcal{I}^{2r}_{{\rho\overline{\nu}}}}\int_{rk+rQ_{\rho\overline{\nu}}}  a\left(\frac{x}{\e}\right)  |\nabla u_\e^{\rho\overline{\nu}}|^2\,dx   d\mathcal{H}^{k_d}(\overline{\nu}) d\rho \\  \label{intermediate} 
&\geq& \frac{1-s}{d} \frac{\sigma_{d-1}}{\mathcal{H}^{k_d}(V)} \int_0^1  \frac{r^{2(1-s)}}{\rho^{-1+2s}}\int_V \int_{\Omega'}  a\left(\frac{x}{\e}\right)  |\nabla u_\e^{\rho\overline{\nu}}|^2\,dx   d\mathcal{H}^{k_d}(\overline{\nu}) d\rho,
\end{eqnarray}
where $\Omega'$ is any open subset well contained in $\Omega$ that we may assume to be well contained in $\displaystyle \bigcup_{k \in \mathcal{I}^{2r}_{\rho\overline{\nu}}} rk+rQ_{\rho\overline{\nu}}$ for $\e$ sufficiently small.

\smallskip

In the final part of the proof, we average the auxiliary functions obtained by discretization weighting them with the appropriate kernel. 

For every $\e>0$ sufficiently small we define the probability measure $\mu_\e$ on $[0,1] \times V$ by
\[
d\mu_\e(\rho, \overline{\nu}):= 2(1-s)\rho^{1-2s}\mathcal{L}^1 \otimes \mathcal{H}^{k_d} [\mathcal{H}^{k_d}(V)]^{-1}  ;
\]
then we rewrite \eqref{intermediate} as
\begin{eqnarray}\label{disega}\nonumber&&
\hskip-3cm \frac{\sigma_{d-1}}{2d} \int_{[0,1]\times V} r^{2(1-s)} \int_{\Omega'} a\left(\frac{x}{\e}\right)   |\nabla u_\e^{\rho\overline{\nu}}|^2\,dx\,d\mu_\e(\rho,\overline{\nu}) 
\\ \nonumber
&=& \frac{\sigma_{d-1}}{2d} \int_{\Omega'} r^{2(1-s)} a\left(\frac{x}{\e}\right) \int_{[0,1]\times V}   |\nabla u_\e^{\rho\overline{\nu}}|^2\,   d\mu_\e(\rho,\overline{\nu})\,dx \\ \label{liminf7}
&\geq& \frac{\sigma_{d-1}}{2d} \int_{\Omega'} r^{2(1-s)} a\left(\frac{x}{\e}\right) \Bigl|\int_{[0,1]\times V}  \nabla u_\e^{\rho\overline{\nu}}\,   d\mu_\e(\rho, \overline{\nu})\Bigr|^2\,dx
\end{eqnarray}
by Jensen's inequality.

We set 
$$
\overline u_\e(x):=\int_{[0,1]\times V}u^{\rho\overline{\nu}}_\e(x)\, d\mu_\e(\rho, \overline{\nu}),
$$
and note that, since $u_\e^{\rho\overline{\nu}}$ is continuous on $\Omega'$, this function belongs to $W^{1,2}(\Omega')$ with
\[
\nabla \overline u_\e(x)=\int_{[0,1]\times V}\nabla u^{\rho\overline{\nu}}_\e(x)\, d\mu_\e(\rho, \overline{\nu}).
\]
As a consequence, \eqref{liminf7} reduces to  
$$
\frac{\sigma_{d-1}}{ 2d}\int_{\Omega'} r^{2(1-s)} a\Bigl(\frac{x}{\e}\Bigr) |\nabla  \overline u_\e(x)|^2dx
$$
and our final goal is to show that
\begin{equation*}
    \liminf_{\e\to0} \frac{\sigma_{d-1}}{2d}\int_{\Omega'} a\Bigl(\frac{x}{\e}\Bigr) |\nabla  \overline u_\e(x)|^2dx \,\,\, \geq \,\,\, \frac{\sigma_{d-1}}{2d}\int_{\Omega'} \langle A_{\rm hom}\nabla  u,\nabla u\rangle\,dx,
\end{equation*}
where we used that $r^{2(1-s)}\to1$, as already stated in \eqref{eq:power}.

To check this it is sufficient to prove that $\overline{u}_\e \to u$ in $L^2(\Omega')$.
Indeed, this implies that $u$ is the $L^2(\Omega')$-limit of a bounded sequence in $W^{1,2}(\Omega')$, so that $u\in\ W^{1,2}(\Omega')$ and the conclusion follows by the liminf inequality provided by the Homogenization Theorem (see \cite[Theorem 14.7]{BDF}). 

First we prove the convergence under the additional assumption that $(u_\e)_\e$ is bounded in $L^\infty(\Omega')$.
If this is the case, it suffices to prove that $\overline{u}_\e-u_\e \to 0$ in $L^1(\Omega')$, indeed, since  $(\overline{u}_\e)_\e$ is bounded by construction in $L^{\infty}(\Omega')$, we conclude by interpolation.

\smallskip

As $\Omega' \subseteq \displaystyle\bigcup_{k\in \mathcal{I}^{2r}_{\rho\overline{\nu}}}rk+rQ_{\rho\overline{\nu}}$, we have
\begin{eqnarray} \nonumber
    \|\overline{u}_\e-u_\e\|_{L^1(\Omega')} &=& \int_{\Omega'} \biggl| u_\e(x) -\int_{[0,1]\times V}u_\e^{\rho\overline{\nu}}(x) d\mu_\e(\rho,\overline{\nu})\biggr|\,dx \\ \nonumber
    & \leq & \int_{[0,1]\times V} \int_{\Omega'} |u_\e(x) - u_\e^{\rho\overline{\nu}}(x)|\,dx\,d\mu_\e(\rho,\overline{\nu}) \\ \nonumber 
    & \leq & \int_{[0,1]\times V} \sum_{k\in \mathcal{I}^{2r}_{\rho\overline{\nu}}} \int_{rk+rQ_{\rho\overline{\nu}}} |u_\e(x) - u_\e^{\rho\overline{\nu}}(x)|\,dx\, d\mu_\e(\rho,\overline{\nu}) \\ \label{L1conv1}
    &\leq&  \int_{[0,1]\times V} \sum_{k\in \mathcal{I}^{2r}_{\rho\overline{\nu}}} \int_{rk+rQ_{\rho\overline{\nu}}} |u_\e(x) - u_\e^{\rho\overline{\nu}}(rk)|\,dx\,d\mu_\e(\rho,\overline{\nu}) \\ \label{L1conv2}
    &+& \int_{[0,1]\times V} \sum_{k\in \mathcal{I}^{2r}_{\rho\overline{\nu}}} \int_{rk+rQ_{\rho\overline{\nu}}} |u^{\rho\overline{\nu}}_\e(x) - u_\e^{\rho\overline{\nu}}(rk)|\,dx\,d\mu_\e(\rho,\overline{\nu}).
\end{eqnarray}
We first prove that \eqref{L1conv1} tends to $0$ as $\e$ vanishes. Applying H\"older's inequality and a rescaled version of Poincaré-Wirtinger's inequality (see \cite[Theorem 6.33]{Leoni2023AFC}) we get
\begin{eqnarray*}\label{attempt4} && \hskip-2.5cm\nonumber
    \sum_{k\in \mathcal{I}^{2r}_{\rho\overline{\nu}}} \int_{rk+rQ_{\rho\overline{\nu}}} |u_\e(x) - u_\e^{\rho\overline{\nu}}(rk)|\,dx \\ \nonumber
    &\leq& |r\rho|^{\frac{d}{2}} \sum_{k\in \mathcal{I}^{2r}_{\rho\overline{\nu}}} \Bigl( \int_{rk+rQ_{\rho\overline{\nu}}} |u_\e(x) - u_\e^{\rho\overline{\nu}}(rk)|^2\,dx \Bigr)^{\frac{1}{2}} \\ \nonumber 
    &\leq& P|r\rho|^{\frac{d}{2}+s} \sum_{k\in \mathcal{I}^{2r}_{\rho\overline{\nu}}} \Bigl(\iint_{(rk+rQ_{\rho\overline{\nu}})\times(rk+rQ_{\rho\overline{\nu}})}\frac{|u_\e(x)-u_\e(y)|^2}{|x-y|^{d+2s}}\,dxdy\Bigr)^{\frac{1}{2}},
    \end{eqnarray*} 
where $P$ is the Poincarè-Wirtinger constant for the  $d$-dimensional unit cube. Then we make use of the concavity of $x\mapsto x^{\frac{1}{2}}$ so that   
\begin{eqnarray*} && \hskip-1.8cm \nonumber
P|r\rho|^{\frac{d}{2}+s} \# \mathcal{I}^{2r}_{\rho\overline{\nu}} \sum_{k\in \mathcal{I}^{2r}_{\rho\overline{\nu}}}\frac{1}{\# \mathcal{I}^{2r}_{\rho\overline{\nu}}} \Bigl(\iint_{(rk+rQ_{\rho\overline{\nu}})\times(rk+rQ_{\rho\overline{\nu}})}\frac{|u_\e(x)-u_\e(y)|^2}{|x-y|^{d+2s}}\,dxdy\Bigr)^{\frac{1}{2}} \\ \nonumber
&\leq& P|r\rho|^{\frac{d}{2}+s} \# \mathcal{I}^{2r}_{\rho\overline{\nu}} \Bigl( \sum_{k\in \mathcal{I}^{2r}_{\rho\overline{\nu}}} \frac{1}{\# \mathcal{I}^{2r}_{\rho\overline{\nu}}} \iint_{(rk+rQ_{\rho\overline{\nu}})\times(rk+rQ_{\rho\overline{\nu}})}\frac{|u_\e(x)-u_\e(y)|^2}{|x-y|^{d+2s}}\,dxdy  \Bigr)^{\frac{1}{2}} \\ \nonumber
&\leq& P|r\rho|^{\frac{d}{2}+s}  (\# \mathcal{I}^{2r}_{\rho\overline{\nu}})^{\frac{1}{2}} \Bigl( \iint_{\Omega \times \Omega}\frac{|u_\e(x)-u_\e(y)|^2}{|x-y|^{d+2s}}\,dxdy \Bigr)^{\frac{1}{2}} \\ \nonumber
&\leq& P|r\rho|^{\frac{d}{2}+s} |2 r\rho|^{-\frac{d}{2}}|\Omega|^{\frac{1}{2}}[u_\e]_{W^{s,2}(\Omega)} \\ \nonumber
&=& |r\rho|^s\,2^{-\frac{d}{2}} |\Omega|^{\frac{1}{2}} [u_\e]_{W^{s,2}(\Omega)}.
\end{eqnarray*} 
After integration we obtain that \eqref{L1conv1} is bounded above by
\begin{equation*}
r^s\,2^{1-\frac{d}{2}}|\Omega|^{\frac{1}{2}}(1-s) [u_\e]_{W^{s,2}(\Omega)} \int_0^1 \rho^{1-s}d\rho\,\,\, = \,\,\,r^s\,2^{1-\frac{d}{2}}|\Omega|^{\frac{1}{2}}(1-s) [u_\e]_{W^{s,2}(\Omega)} \frac{1}{2-s}.
\end{equation*}
Then we use that $[u_\e]_{W^{s,2}(\Omega)}(1-s)\leq F_\e(u_\e)/\alpha$, which is uniformly bounded in $\e$ by assumption, and that $r^s\to0$ to deduce that the above term vanishes.

\smallskip

We treat \eqref{L1conv2} with a similar argument. By the Cauchy-Schwarz inequality we have
\begin{eqnarray*} && \hskip-3cm \nonumber
    \int_{[0,1]\times V} \sum_{k\in \mathcal{I}^{2r}_{\rho\overline{\nu}}} \int_{rk+rQ_{\rho\overline{\nu}}} |u^{\rho\overline{\nu}}_\e(x) - u_\e^{\rho\overline{\nu}}(rk)|\,dx\,d\mu_\e(\rho,\overline{\nu}) \\ \nonumber
    &\leq& \int_{[0,1]\times V} \sum_{k\in \mathcal{I}^{2r}_{\rho\overline{\nu}}} \int_{rk+rQ_{\rho\overline{\nu}}} |\nabla u^{\rho\overline{\nu}}_\e(x)||x-rk|\,dx\,d\mu_\e(\rho,\overline{\nu}) \\ \nonumber
    &\leq& \sqrt{d}\int_{[0,1]\times V} |r\rho| \sum_{k\in \mathcal{I}^{2r}_{\rho\overline{\nu}}} \int_{rk+rQ_{\rho\overline{\nu}}} |\nabla u^{\rho\overline{\nu}}_\e(x)|\,dx\,d\mu_\e(\rho,\overline{\nu})
\end{eqnarray*}
and by H\"older's inequality and the concavity of $x \mapsto x^\frac{1}{2}$, this is bounded above by
\begin{eqnarray*} && \hskip-1.2cm \nonumber
\hskip-0.6cm \sqrt{d}\int_{[0,1]\times V} |r\rho|^{\frac{d}{2}+1}  \sum_{k\in \mathcal{I}^{2r}_{\rho\overline{\nu}}} \biggl(\int_{rk+rQ_{\rho\overline{\nu}}} |\nabla u^{\rho\overline{\nu}}_\e(x)|^2\,dx\biggr)^{\frac{1}{2}} d\mu_\e(\rho,\overline{\nu}) \\ \nonumber
&=& \sqrt{d}\int_{[0,1]\times V} |r\rho|^{\frac{d}{2}+1}  \#\mathcal{I}^{2r}_{\rho\overline{\nu}} \sum_{k\in \mathcal{I}^{2r}_{\rho\overline{\nu}}} \frac{1}{\#\mathcal{I}^{2r}_{\rho\overline{\nu}}} \biggl(\int_{rk+rQ_{\rho\overline{\nu}}} |\nabla u^{\rho\overline{\nu}}_\e(x)|^2\,dx\biggr)^{\frac{1}{2}} d\mu_\e(\rho,\overline{\nu}) \\ \nonumber
&\leq& \sqrt{d}\int_{[0,1]\times V} |r\rho|^{\frac{d}{2}+1}  (\#\mathcal{I}^{2r}_{\rho\overline{\nu}})^{\frac{1}{2}} \biggl(\sum_{k\in \mathcal{I}^{2r}_{\rho\overline{\nu}}} \int_{rk+rQ_{\rho\overline{\nu}}} |\nabla u^{\rho\overline{\nu}}_\e(x)|^2\,dx\biggr)^{\frac{1}{2}} d\mu_\e(\rho,\overline{\nu}) \\ \nonumber
&\leq& r\sqrt{d \, 2^{-d} |\Omega|}\int_{[0,1]\times V}  \biggl(\sum_{k\in \mathcal{I}^{2r}_{\rho\overline{\nu}}} \int_{rk+rQ_{\rho\overline{\nu}}} |\nabla u^{\rho\overline{\nu}}_\e(x)|^2\,dx\biggr)^{\frac{1}{2}} d\mu_\e(\rho, \overline{\nu}), 
\end{eqnarray*}
which by Jensen's inequality is less than or equal to
\begin{equation}\label{liminf8}
r \sqrt{d\, 2^{-d}|\Omega|} \Bigl(  \int_{[0,1]\times V} \sum_{k\in \mathcal{I}^{2r}_{\rho\overline{\nu}}}\int_{rk+rQ_{\rho\overline{\nu}}} |\nabla u_\e^{\rho\overline{\nu}}(x)|^2 dx\, d\mu_\e(\rho,\overline{\nu}) \Bigr)^{\frac{1}{2}}.     
\end{equation}
Now we write
\begin{eqnarray} && \nonumber \hskip-1cm \nonumber \int_{[0,1]\times V} \sum_{k\in \mathcal{I}^{2r}_{\rho\overline{\nu}}}\int_{rk+rQ_{\rho\overline{\nu}}} |\nabla u_\e^{\rho\overline{\nu}}(x)|^2 dx\, d\mu_\e(\rho,\overline{\nu})
\\ \nonumber
&=& \frac{2(1-s)}{\mathcal{H}^{k_d}(V)} \int_0^1  \frac{1}{\rho^{-1+2s}}\int_V \sum_{k\in \mathcal{I}^{2r}_{{\rho\overline{\nu}}}} \int_{rk+rQ_{\rho\overline{\nu}}} |\nabla u_\e^{\rho\overline{\nu}}|^2\,dx   d\mathcal{H}^{k_d}(\overline{\nu}) d\rho \\ \nonumber
&\leq& \hskip-0.3cm \Bigl(\frac{2d}{\alpha\sigma_{d-1}}\Bigr)\frac{1-s}{d} \frac{\sigma_{d-1}}{\mathcal{H}^{k_d}(V)} \int_0^1  \frac{r^{2(1-s)}}{\rho^{-1+2s}}\int_V \sum_{k\in \mathcal{I}^{2r}_{{\rho\overline{\nu}}}} a\left(\frac{rk}{\e}\right) \int_{rk+rQ_{\rho\overline{\nu}}} |\nabla u_\e^{\rho\overline{\nu}}|^2\,dx   d\mathcal{H}^{k_d}(\overline{\nu})d\rho
\end{eqnarray}
and we note that this last term equals \eqref{midterm} (up to a constant factor), therefore, as a byproduct of the previous computations, it is bounded above by   
\[
\frac{2d}{\alpha\sigma_{d-1}}F_\e(u_\e)+o_\e(1),
\]
which is uniformly bounded for small $\e$. We conclude that \eqref{liminf8} vanishes since $r\to0$.

Finally, we remove the boundedness condition by a truncation argument. Let $u^M:= (u \land M) \lor -M$, and note that $u_\e\to u$ in $L^2(\Omega)$ implies $u_\e^M\to u^M$ in $L^2(\Omega)$. Since $F_\e(u_\e) \geq F_\e(u_\e^M)$, we have
\[
\liminf_{\e\to0} F_\e(u_\e) \  \geq \  \liminf_{\e\to0} F_\e(u^M_\e)\  \geq \ \frac{\sigma_{d-1}}{2d}\int_{\Omega} \langle A_{\rm hom}\nabla u^M, \nabla u^M\rangle\,dx,
\]
and then
\[
\sup_M \|\nabla u^M\|_{L^2(\Omega)} \  \leq \  \frac{2d}{\alpha\sigma_{d-1}}\sup_\e F_\e(u_\e)\  < \  \ +\infty,
\]
so that $u\in W^{1,2}(\Omega)$ with $\nabla u^M \to \nabla u$. The proof is concluded once we let $M\to+\infty$.

\subsection{\bf Limsup inequality}

To prove the $\Gamma$-$\limsup$ inequality, we first study the case where $u$ is piecewise-affine. We then recover the inequality on the whole space $W^{1,2}(\Omega)$ by means of a density argument. For the sake of clarity, we first deal with the case where $u$ is affine.

\smallskip

\textit{Step $1$}. We first suppose that $u(x)=\langle z,x\rangle$, for some $z\in\mathbb{R}^d$. We recall that by the classical homogenization formula for $A_{\rm hom}$ (see \cite[Theorem 14.7]{BDF}) and by density, it holds
\be\label{eq:defphi}
\langle A_{\rm hom}z, z\rangle =\inf\Bigl\{\int_{(0,1)^d}a(y)\lvert\nabla{\varphi}(y)+z\rvert^2dy\,:\,\varphi\in C^\infty(\Rd),\,\,1\text{-periodic}\Bigr\}.
\ee

\noindent Fix $\delta>0$ and let $\varphi\in C^\infty(\Rd)$ be a $1$-periodic function such that
\be\label{eq:phidelta}
\int_{(0,1)^d}a(y)\lvert\nabla\varphi(y)+z\rvert^2 dy < \langle A_{\rm hom}z,z\rangle + \delta;
\ee
we claim that $u_\e(x)=\langle z,x\rangle+\e\varphi\left(\frac{x}{\e}\right)$ is an approximate recovery sequence for $u$, i.e., it is such that
\be\label{eq:claimaffine}
\limsup_{\e\to0} F_\e(u_\e) \leq\frac{\sigma_{d-1}}{2d}\int_\Omega\langle A_{\rm hom}z,z\rangle dx+\delta.
\ee
This, combined with a diagonal argument and the arbitrariness of $\delta>0$, will lead to the limsup inequality.

In order to prove \eqref{eq:claimaffine}, we take advantage of Lemma \ref{lemma:locality} with $r_\e=1$, and consider the first-order Taylor expansion of $\varphi$ to get
\begin{eqnarray*} \label{eq:Feueaffred} \nonumber
&&\hskip-1cm F_\e(u_\e) +o_\e(1)\\ \nonumber
&=& (1-s)\iint_{\Om\times\Om} a\left(\frac{x}{\e}\right)\frac{\lvert \langle z,x-y\rangle+\e\varphi\left(\frac{x}{\e}\right)-\e\varphi\left(\frac{y}{\e}\right)\rvert^2}{\lvert x-y\rvert^{d+2s}}dxdy \\
&=& (1-s)\iint_{\Om\times\Om\cap\{\lvert x-y\rvert\leq 1\}}a\left(\frac{x}{\e}\right)\frac{|\langle z + \nabla{\varphi}\left(\frac{x}{\e}\right),x-y\rangle+\frac{1}{\e}R(|x-y|)\rvert^2}{\lvert x-y\rvert^{d+2s}}dxdy,
\end{eqnarray*} 
where the function $R(|x-y|)$ is the remainder in Lagrange form of $\varphi$ and it is such that
\be \label{eq:taylorphi}
R(|x-y|)\leq d^2||\nabla^2\varphi||_{\rm \infty}|x-y|^2.
\ee

By the inequality $|a+b|^2 \leq (1+\eta)a^2+(1+1/\eta)b^2$, with $\eta>0$, and by the change of variables $\xi:=y-x$, we deduce
\begin{eqnarray} \label{eq:goodternaffinered}
F_\e(u_\e)+o_\e(1) &\leq & (1+\eta)(1-s)\int_\Omega \int_{B_1(0)} a\left(\frac{x}{\e}\right)\frac{|\langle z + \nabla\varphi \left(\frac{x}{\e}\right), \xi \rangle|^2}{|\xi|^{d+2s}} d\xi dx \\ \label{eq:vanishingtermaffinered}
& & +\left(1+\frac{1}{\eta}\right) \frac{1-s}{\e ^2}\int_\Omega \int_{B_1(0)} a\left(\frac{x}{\e}\right) \frac{ R(|\xi|)^2}{|\xi|^{d+2s}}d\xi dx.  
\end{eqnarray}
For a fixed $\eta>0$, the term in \eqref{eq:vanishingtermaffinered} vanishes as $\e$ tends to zero. Indeed, by assumption, the term $\frac{1-s}{\e^2}$ vanishes as $\e\to0$, and using \eqref{eq:bounds} we have that
 \begin{eqnarray}     \nonumber
\int_\Omega\int_{B_1(0)} a\left(\frac{x}{\e}\right) \frac{ R(|\xi|)^2}{|\xi|^{d+2s}}d\xi dx & \leq & d^4\beta|\Om|||\nabla^2\varphi||_{\rm \infty} \int_{B_1(0)} \frac{1}{|\xi|^{d+2s-4}}d\xi \\ \label{eq:integralesolito}
 & = & d^4\sigma_{d-1}\beta|\Om|||\nabla^2\varphi||_{\rm \infty}\int_0^1\rho^{3-2s}d\rho<+\infty.
 \end{eqnarray}

As for the term in  \eqref{eq:goodternaffinered}, note that by the symmetry of the integrand, it holds
\begin{eqnarray}\nonumber 
\int_{B_1(0)}\frac{| \langle z + \nabla{\varphi}\left(\frac{x}{\e}\right), \xi\rangle |^2}{\lvert \xi\rvert^{d+2s}}d\xi &=& \int_0^1 \int_{\partial B_\rho(0)} \frac{\left|\langle z+\nabla\varphi\left(\frac{x}{\e}\right) , \nu \rangle\right|^2} {\rho^{d+2s}} d\Hd\nu d\rho \\ \nonumber
&=& \int_0^1 \frac{1}{\rho^{d+2(s-1)}}\frac{1}{d}\int_{\partial B_\rho(0)} \left|z+\nabla\varphi\left(\frac{x}{\e}\right)\right|^2 d\Hd\nu d\rho \\ \nonumber
&=& \frac{\sigma_{d-1}}{d}\left|z+\nabla\varphi\left(\frac{x}{\e}\right)\right|^2 \int_0^1 \rho^{1-2s}d\rho \\ \label{eq:integrandgoodaffinered}
&=& \frac{\sigma_{d-1}}{2d}\left|z+\nabla\varphi\left(\frac{x}{\e}\right)\right|^2 \frac{1}{(1-s)}.
\end{eqnarray}
 Recalling that $\varphi$ and $a$ are $1$-periodic, we may substitute \eqref{eq:integrandgoodaffinered} in \eqref{eq:goodternaffinered} to get
\begin{eqnarray}\nonumber 
&& \hskip-3cm (1+\eta)(1-s)\int_\Omega \int_{B_1(0)} a\left(\frac{x}{\e}\right)\frac{|\langle z + \nabla\varphi \left(\frac{x}{\e}\right), \xi \rangle|^2}{|\xi|^{d+2s}} d\xi dx 
\\ \nonumber
&=&(1+\eta)\frac{\sigma_{d-1}}{2d}\int_\Omega a\left(\frac{x}{\e}\right)\left\lvert\nabla{\varphi}\left(\frac{x}{\e}\right)+z\right\rvert^2dx \\ \nonumber
&\leq& (1+\eta)\frac{\sigma_{d-1}}{2d}\frac{|\Om|}{\e^d}\int_{(0,\e)^d}a\left(\frac{x}{\e}\right)\left|\nabla\varphi\left(\frac{x}{\e}\right)+z\right|^2dx\\ \label{estimateperiod}
&=& (1+\eta)\frac{\sigma_{d-1}}{2d}\lvert\Om\rvert\int_{(0,1)^d} a(y)\left\lvert\nabla{\varphi}\left(y\right)+z\right\rvert^2dy,
\end{eqnarray}
which, in view of \eqref{eq:phidelta}, implies
\be\nonumber
\limsup_{\e\to 0}F_\e(u_\e)\leq(1+\eta)\left(\frac{\sigma_{d-1}}{2d}\int_\Omega\langle A_{\rm hom}z,z\rangle dx+\delta\right),
\ee
 and we conclude by letting $\eta$ tend to $0$.

\bigskip

\textit{Step $2$}. We now suppose that the function $u$ be piecewise-affine in $\Om$. More precisely, we suppose that there exists a finite family of $d$-simplices $(\Delta)_{i\in I}$ covering $\Omega$, vectors $(z_i)_{i\in I}\subseteq \Rd$, and constants $(m_i)_{i\in I}\subseteq \mathbb{R}$ such that $u(x)=\langle x,z_i\rangle+m_i$ for every $x\in \Delta_i\cap\Om$ and for every $i\in I$. To simplify the notation, from now on we will always write $\Delta_i$ in place of $\Delta_i\cap\Omega$.

Similarly to what we have done in the previous case, for a fixed $\delta>0$ and for every $i\in I$, we consider $\varphi_i\in C^\infty_c(\Rd)$ 
a $1$-periodic function satisfying 
\be\label{eq:deltacardinalita}
\int_{(0,1)^d}a(y)\lvert\nabla\varphi_i(y)+z_i\rvert^2dy<\langle A_{\rm hom}z_i,z_i\rangle+\frac{\delta}{\#I}.
\ee
Having in mind the construction of the recovery sequence performed in the previous case, we aim at constructing an approximate recovery sequence for $u$ by perturbing, in a suitable way, $u$ with the functions $\varphi_i$. To this end,
for every $\Delta_i$, we consider the subset defined by
\begin{equation*}
\Omi:=\left\{ x\in\Om:\,\,\text{dist}(x,\Delta_i^c)>\e\right\}.
\end{equation*}

For fixed $\e>0$ and $i\in I$, we consider a positive function $\psi^i_\e\in C^\infty_c(\mathbb{R}^d)$, satisfying the following conditions
\begin{subequations}
  \begin{empheq}[left=\empheqlbrace]{align}
 \nonumber &\ \psi^i_\e(x)=0\qquad\qquad \text{for } x\notin \Omi\\
 \nonumber &\ \psi^i_\e(x)=1\qquad\qquad \text{for } x\in \Omid\\
  \label{eq:psibound}  &\ 0\leq\psi_\e^i(x)\leq 1 \qquad\,  \text{for }x\in \mathbb{R}^d\\  \label{eq:gradientbound}&\ \|\nabla\psi^i_\e\|_\infty\leq\frac{2}{\e}& \\
\label{eq:hessianbound} &\ \|\nabla^2\,\psi^\textit{i}_\e\|_\infty\leq\frac{4}{\e^2}& \,,
  \end{empheq}
\end{subequations}
where $\nabla^2\psi_\e^i$ is the Hessian matrix of $\psi_\e^i$. 

We claim that an approximate recovery sequence for $u$ is given by 
\begin{equation*}
u_\e(x)=\sum_{i\in I}\left(\langle z_i,x\rangle+\e\varphi_i\left(\frac{x}{\e}\right)\psi_\e^i(x)+m_i\right)\chi_{\Delta_i}(x),
\end{equation*}
where $\chi_{\Delta_i}$ denotes the characteristic function of $\Delta_i$. Note that $u_\e$ is a continuous function, since it coincides with $u$ on the boundary of each $\Delta_i$.

Let $A:={\Om\times\Om\cap\{\lvert x-y\rvert\leq \e\}}$. Thanks to Lemma \ref{lemma:locality} and Remark \ref{re:remark}, it is enough to show that 
\begin{equation} \nonumber 
\limsup_{\e\to 0}(1-s)\iint_A a\left(\frac{x}{\e}\right)\frac{\lvert u_\e(x)-u_\e(y)\rvert^2}{\lvert x-y\rvert^{d+2s}}dxdy\leq \frac{\sigma_{d-1}}{2d}\int_{\Om} \langle A_{\rm hom}\nabla{u},\nabla{u}\rangle dx+\delta.
\end{equation}
To make the computations more manageable, we consider a suitable partition of $A$. For $i\in I$, we set
\begin{equation*}
B_i =\Omid\times\Omid\cap\{\lvert x-y\rvert \leq \e\},
\end{equation*}
\begin{equation*}
C_i = \Delta_i\times\Delta_i\cap\{\lvert x-y\rvert\leq \e \text{ and } x \text{ or }y\in \Delta_i\setminus{\Omid}\},
\end{equation*} 
and for $j \in I,\, j \neq i$, we set
\[
D_{ij}=\Delta_i\times\Delta_j\cap\{\lvert x-y\rvert \leq \e\},
\]
so that $A=\bigcup_{i\in I}(B_i\cup C_i)\cup \bigcup_{i\neq j}D_{ij}$.

Taking advantage of this decomposition and of the subadditivity of the $\limsup$, we have
\begin{eqnarray}
&& \nonumber\hskip-2cm \limsup_{\e\to 0}\, (1-s)\iint_A a\left(\frac{x}{\e}\right)\frac{\lvert u_\e(x)-u_\e(y)\rvert^2}{\lvert x-y\rvert^{d+2s}}dxdy\\
&&\label{eq:FunctionalonBi} \leq  \limsup_{\e\to 0}\sum_{i\in I}(1-s)\iint_{B_i} a\left(\frac{x}{\e}\right)\frac{\lvert u_\e(x)-u_\e(y)\rvert^2}{\lvert x-y\rvert^{d+2s}}dxdy\\
&&\label{eq:FunctionalonCi} +\sum_{i\in I}\limsup_{\e\to 0}\,(1-s)\iint_{C_i} a\left(\frac{x}{\e}\right)\frac{\lvert u_\e(x)-u_\e(y)\rvert^2}{\lvert x-y\rvert^{d+2s}}dxdy\\ 
&& \label{eq:FunctionalonDij} +\sum_{i\in I}\sum_{j\in I,\, j\neq i}\limsup_{\e\to 0}\,(1-s)\iint_{D_{ij}} a\left(\frac{x}{\e}\right)\frac{\lvert u_\e(x)-u_\e(y)\rvert^2}{\lvert x-y\rvert^{d+2s}}dxdy.
\end{eqnarray}

We begin by estimating \eqref{eq:FunctionalonBi}. Since for fixed $i\in I$ we have $\psi_\e^i=1$ on $B_i$, by the first-order Taylor expansion of $\varphi_i$, for a fixed $\eta>0$ we get
\begin{eqnarray}\nonumber
   && \hskip-2.5cm  \sum_{i\in I}(1-s)\iint_{B_i}a\left(\frac{x}{\e}\right)\frac{|\langle z_i + \nabla{\varphi}_i\left(\frac{x}{\e}\right),x-y\rangle+\frac{1}{\e}R_i(|x-y|)\rvert^2}{\lvert x-y\rvert^{d+2s}}dxdy\\ 
   &\leq&\label{eq:nonvanishingeta}  \sum_{i\in I}(1+\eta)(1-s)\int_{\Omid}\int_{B_\e(0)} a\left(\frac{x}{\e}\right)\frac{|\langle z_i + \nabla\varphi_i \left(\frac{x}{\e}\right), \xi \rangle|^2}{|\xi|^{d+2s}} d\xi dx
  \\&+&\label{eq:vanishingeta} \sum_{i\in I}\left(1+\frac{1}{\eta}\right) \frac{1-s}{\e ^2}\int_{\Omid} \int_{B_\e(0)} a\left(\frac{x}{\e}\right) \frac{R_i(|\xi|)^2}{|\xi|^{d+2s}}d\xi dx,
\end{eqnarray}
 where $R_i$ is the remainder in Lagrange form of the first-order Taylor expansion of $\varphi_i$, which satisfies   
\be\label{eq:rest}
R_i(|x-y|)\leq d^2 \max_{i\in I}||\nabla^2\,\varphi_i||_{\infty}|x-y|^2=:M|x-y|^2.
\ee
Thus, we may estimate  \eqref{eq:vanishingeta} with 
\be \nonumber
\beta M^2\sum_{i\in I}\left(1+\frac{1}{\eta}\right) \frac{1-s}{\e ^2}\int_{\Delta_i} \int_{B_1(0)} \frac{|\xi|^4}{|\xi|^{d+2s}}d\xi dx,
\ee
which vanishes for $\e \to 0$ in view of the fact that $\frac{1-s}{\e^2}$ goes to zero for $\e\to0$ and of \eqref{eq:integralesolito} (with $\Omega$ replaced by $\Delta_i$).

As for \eqref{eq:nonvanishingeta}, considering larger domains
and taking advantage of the computations leading to \eqref{estimateperiod}, we see that 
\begin{eqnarray*}
    &&\hspace{-2 cm}\displaystyle \sum_{i\in I}(1+\eta)(1-s)\int_{\Omid}\int_{B_\e(0)} a\left(\frac{x}{\e}\right)\frac{|\langle z_i + \nabla\varphi_i \left(\frac{x}{\e}\right), \xi \rangle|^2}{|\xi|^{d+2s}} d\xi dx\\
    &\leq&  \sum_{i\in I}(1+\eta)(1-s)\int_{\Delta_i}\int_{B_1(0)} a\left(\frac{x}{\e}\right)\frac{|\langle z_i + \nabla\varphi_i \left(\frac{x}{\e}\right), \xi \rangle|^2}{|\xi|^{d+2s}} d\xi dx;\\
    &\leq& (1+\eta)\frac{\sigma_{d-1}}{2d}\sum_{i\in I}\lvert\Delta_i\rvert\int_{(0,1)^d} a(y)\left\lvert\nabla{\varphi_i}\left(y\right)+z_i\right\rvert^2dy.
\end{eqnarray*}
Recalling \eqref{eq:deltacardinalita}, we finally conclude
\begin{eqnarray*}
\nonumber
&&\hspace{-3cm}\limsup_{\e \to 0}\sum_{i=1}^n(1+\eta)(1-s)\iint_{B_i}a\left(\frac{x}{\e}\right)\frac{\lvert u_\e(x)-u_\e(y)\rvert^2}{\lvert x-y\rvert^{d+2s}}\\&\leq&\limsup_{\e \to 0}\,(1+\eta)\frac{\sigma_{d-1}}{2d}\sum_{i\in I}\Bigl(\int_{\Delta_i}\langle A_{\rm hom}z_i,z_i\rangle+\frac{\delta}{\#I}\Bigr)\\
&\leq&(1+\eta)\frac{\sigma_{d-1}}{2d}\hspace{-0.2cm}\int_{\Om} \langle A_{\rm hom}\nabla{u},\nabla{u}\rangle dx+\delta,
\end{eqnarray*}
so that letting $\eta \to 0$,  \eqref{eq:FunctionalonBi} is estimated.

We now show that \eqref{eq:FunctionalonCi} vanishes for $\e \to 0$. As $\#I$ is finite and independent on $\e$, it suffices to take into account a single term in the sum.
By  the convexity of $x\mapsto x^2$, we have 
\begin{eqnarray}\nonumber
&&\hspace{-2 cm} (1-s)\iint_{C_i}a\left(\frac{x}{\e}\right)\frac{\lvert u_\e(x)-u_\e(y)\rvert^2}{\lvert x-y\rvert^{d+2s}}dxdy\\ \label{eq:convex1}
&\leq& 3\beta(1-s)\iint_{C_i}\frac{ \lvert z_i\rvert^2\lvert x-y\rvert^2}{\lvert x-y\rvert^{d+2s}}dxdy \\ 
\label{eq:convex2}
  &&  + 3\beta(1-s)\iint_{C_i}\frac{\left\lvert\e\psi^i_\e(y)\left(\varphi_i\left(\frac{x}{\e}\right)-\varphi_i\left(\frac{y}{\e}\right)\right)\right\rvert^2}{\lvert x-y\rvert^{d+2s}}dxdy\\ \label{eq:convex3}
 && +3\beta(1-s)\iint_{C_i}\frac{\left\lvert\e\varphi_i\left(\frac{x}{\e}\right)(\psi^i_\e(x)-\psi^i_\e(y))\right\rvert^2}{\lvert x-y\rvert^{d+2s}}dxdy.
\end{eqnarray}

We start with \eqref{eq:convex1}. Via a change of variables, we see that  
\begin{eqnarray*}\nonumber
  (1-s)\iint_{C_i}\frac{ \lvert z_i\rvert^2\lvert x-y\rvert^2}{\lvert x-y\rvert^{d+2s}}dxdy &\leq& (1-s)\lvert z_i\rvert^2\int_{\Delta_i\setminus{\Omid}}\int_{B_{1}(0)}\frac{\lvert\xi\rvert^2}{\lvert\xi\rvert^{d+2s}}d\xi dx\\
   &\leq&
    (1-s)\rvert\lvert z_i\rvert^2\lvert\Delta_i\setminus{\Omid}\rvert \sigma_{d-1}\int_{B_{1}(0)}\frac{\lvert\xi\rvert^2}{\lvert\xi\rvert^{d+2s}}d\xi \label{eq:radialintegral}\\&=&\frac{\sigma_{d-1}}{2}\lvert z_i\rvert^2\lvert\Delta_i\setminus{\Omid}\rvert,
\end{eqnarray*}
which converges to zero as $\e\to 0$. 

As for \eqref{eq:convex2}, using the first-order Taylor expansions of $\varphi_i$, \eqref{eq:psibound}, \eqref{eq:gradientbound}, and recalling \eqref{eq:rest}, one gets
\begin{eqnarray*}\nonumber
&&\hspace{-0.7cm}(1-s)\iint_{C_i}\frac{\left\lvert\e\psi^i_\e(y)\left(\varphi_i\left(\frac{x}{\e}\right)-\varphi_i\left(\frac{y}{\e}\right)\right)\right\rvert^2}{\lvert x-y\rvert^{d+2s}}dxdy\\ 
&\leq& (1-s)\int_{\Delta_i\setminus{\Omid}}\int_{B_{\e}(0)}\frac{\left\lvert\langle\nabla{\varphi_i}\left(\frac{x}{\e}\right),\xi\rangle+\frac{1}{\e}R_i(|\xi|)\right\rvert^2}{\lvert \xi\rvert^{d+2s}}d\xi dx \\ \label{eq:sommaconO}\hspace{-0.7cm}&\leq&
2(1-s)\Bigl(\int_{\Delta_i\setminus{\Omid}}\int_{B_{\e}(0)}\frac{\left\lvert \langle\nabla{\varphi_i}\left(\frac{x}{\e}\right),\xi\rangle\right\rvert^2}{\lvert \xi\rvert^{d+2s}}d\xi dx+\frac{M^2}{\e^2}\int_{\Delta_i\setminus{\Omid}}\int_{B_{\e}(0)}\frac{\lvert\xi\rvert^4}{\lvert \xi\rvert^{d+2s}}d\xi dx\Bigr).
\end{eqnarray*}
The first term can be estimated with 
\begin{eqnarray*}
 &&\hskip-2cm(1-s)\int_{\Delta_i\setminus{\Omid}}\int_{B_{\e}(0)}\frac{\lvert\langle\nabla{\varphi_i}\left(\frac{x}{\e}\right),\xi\rangle\rvert^2}{\lvert \xi\rvert^{d+2s}}d\xi dx\\  \nonumber
&\leq& (1-s)\lvert\lvert \nabla{\varphi_i}\rvert\rvert_\infty^2\int_{\Delta_i\setminus{\Omid}}\int_{B_{\e}(0)}\frac{\lvert\xi\rvert^2}{\lvert \xi\rvert^{d+2s}}d\xi dx\\ \nonumber
&=& (1-s)\lvert\lvert \nabla{\varphi_i}\rvert\rvert_\infty^2\lvert\Delta_i\setminus\Omid\rvert\int_{B_1(0)}\frac{|\xi|^2}{|\xi|^{d+2s}} \\ \nonumber 
&\leq& \frac{\sigma_{d-1}}{2}\lvert\lvert\nabla{\varphi_i}\rvert\rvert_\infty^2|\Delta_i\setminus\Omid|\,,
\end{eqnarray*}
while for the second term we have 
\begin{eqnarray}\label{eq:radialcomputation}
     \frac{1-s}{\e^2}\int_{\Delta_i\setminus{\Omid}}\int_{B_{\e}(0)}\frac{\lvert\xi\rvert^4 }{\lvert \xi\rvert^{d+2s}}d\xi dx
     \leq   \frac{(1-s)}{\e^2}|\Delta_i\setminus\Omid|\frac{\sigma_{d-1}}{4-2s},
\end{eqnarray}
and therefore both terms vanish as $\e$ tends to $0$.

In order to estimate \eqref{eq:convex3}, we denote by $\Tilde{R}_i$ the Lagrange remainders of the first-order Taylor expansion of $\psi_\e^i$ and we note that by \eqref{eq:hessianbound} it follows that
\be
\Tilde{R}_i(|x-y|)\leq \frac{4d^2}{\e^2}|x-y|^2.
\ee

\noindent
Using the first-order Taylor expansion of $\psi^i_\e$, we conclude that
\begin{eqnarray}\nonumber&& \hspace{-2.5 cm}
    (1-s)\iint_{C_i}\frac{\left\lvert\e\varphi_i\left(\frac{x}{\e}\right)(\psi^i_\e(x)-\psi^i_\e(y))\right\rvert^2}{\lvert x-y\rvert^{d+2s}}dxdy\\\nonumber &&
    \leq(1-s)\lvert\lvert\varphi_i\rvert\rvert_\infty^2 \int_{\Delta_i\setminus{\Omid}}\int_{B_{\e}(0)}\frac{\lvert\e\langle\nabla{\psi^i_\e(x)},\xi\rangle+\e \Tilde{R}(|x-y|)\rvert^2}{\lvert\xi\rvert^{d+2s}}d\xi dx\\ \label{eq:usualC} &&
    \leq 2 (1-s)\lvert\lvert\varphi_i\rvert\rvert_\infty^2\Bigl(\int_{\Delta_i\setminus{\Omid}}\int_{B_{\e}(0)}\frac{\lvert\e\langle\nabla{\psi^i_\e(x)},\xi\rangle\rvert^2}{\lvert\xi\rvert^{d+2s}}d\xi dx\\ \label{eq:vanishC1} &&
    \qquad \qquad  \qquad \qquad \qquad  \qquad \qquad + \int_{\Delta_i\setminus{\Omid}}\int_{B_{\e}(0)}\frac{16d^4}{\e^2}\frac{\lvert\xi\rvert^4}{\lvert\xi\rvert^{d+2s}}d\xi dx\biggr).
\end{eqnarray}
In view of \eqref{eq:radialcomputation}, we infer that the term in  \eqref{eq:vanishC1} converges to zero. 

As for \eqref{eq:usualC}, we may use  \eqref{eq:gradientbound} to obtain that
\begin{eqnarray*}&&\nonumber
\hspace{-1 cm}(1-s)\int_{\Delta_i\setminus{\Omid}}\int_{B_{\e}(0)}\frac{\lvert\e\langle\nabla{\psi^i_\e(x)},\xi\rangle\rvert^2}{\lvert\xi\rvert^{d+2s}}d\xi dx\\
&&  \leq 2 (1-s)|\Delta_i\setminus{\Omid}|\int_{B_1(0)}\frac{|\xi|^2}{|\xi|^{d+2s}}d\xi = \sigma_{d-1}|\Delta_i\setminus{\Omid}|,
\end{eqnarray*}
so that we have finally shown that \eqref{eq:FunctionalonCi} tends to $0$ for $\e \to 0$.

\smallskip

We conclude by proving that \eqref{eq:FunctionalonDij} vanishes for $\e\to 0$. Also in this case, since $\#I$ is finite and independent of $\e$, it suffices to take into account just one summand.

Note that if $(x,y)\in D_{ij}$, then $x\in\Delta_i\setminus{\Omi}$ and $y\in\Delta_j\setminus\Delta_j^{\e}$, so that $\psi^i_\e(x)=\psi^j_\e(y)=0.$ 

\noindent Thus, taking advantage of the Lipschitz continuity of $u$, we get
\begin{eqnarray*} \nonumber
    \hspace{-0.5 cm}(1-s)\iint_{D_{ij}}a\left(\frac{x}{\e}\right)\frac{\lvert u_\e(x)-u_\e(y)\rvert^2}{\lvert x-y\rvert^{d+2s}} &\leq& \beta(1-s)\int_{\Delta_i\setminus{\Omi}}\int_{B_{\e}(0)}\frac{\lvert u(x+\xi)-u(x)\rvert^2}{\lvert\xi\rvert^{d+2s}}d\xi dx \\
 &\leq& \beta\lvert\lvert\nabla{u}\rvert\rvert_\infty(1-s)\int_{\Delta_i\setminus{\Omi}}\int_{B_{1}(0)}\frac{\lvert\xi\rvert^2}{\lvert\xi\rvert^{d+2s}}d\xi dx\\
 &=& \beta\lvert\lvert\nabla{u}\rvert\rvert_\infty(1-s) |\Delta_i\setminus{\Omi}| \sigma_{d-1} \frac{1}{2-2s} \\
 &\leq& \beta\lvert\lvert\nabla{u}\rvert\rvert_\infty\frac{\sigma_{d-1}}{2} |\Delta_i\setminus{\Omi}|
\end{eqnarray*}
and this last term vanishes for $\e \to 0$ deducing that \eqref{eq:FunctionalonDij} tends to $0$.

\smallskip

\textit{Step $3$.} The functional $F_{\rm\hom}$ defined in \eqref{mainthm}, which has been proven to be the desired $\Gamma$-limit on piecewise-affine functions, is continuous in the strong topology of $W^{1,2}(\Omega)$. Since $\Omega$ is an open bounded set with Lipschitz boundary, 
every function $u\in W^{1,2}(\Omega)$ may be extended to $\overline{u}\in W^{1,2}_0(B_R)$ via the standard extension operator and a cut-off function.
It is a known fact that a dense class in $W^{1,2}_0(B_R)$ is the class of piecewise-affine functions, namely, functions which are piecewise-affine on a finite family of simplices covering $\overline{B}_R$ (see \cite[Proposition 2.1]{EkelandConvex}). Therefore, the class of piecewise-affine functions discussed in Step $2$ is dense in $W^{1,2}_0(B_R)$ and a fortiori in $W^{1,2}(\Omega)$. Having proved that the $\Gamma$-$\limsup$ inequality holds on a dense subset of $W^{1,2}(\Omega)$, we conclude that it holds for any $u\in W^{1,2}(\Omega)$ (see for instance \cite[Remark 2.8]{handbook}).

\medskip
\noindent \textbf{Acknowledgements.}
 This paper is based on work supported by the National Research Project (PRIN  2017BTM7SN) 
 ``Variational Methods for Stationary and Evolution Problems with Singularities and 
 Interfaces", funded by the Italian Ministry of University and Research. 
Andrea Braides and Davide Donati are members of GNAMPA, INdAM.

\bibliographystyle{plain}
\bibliography{References}
\end{document}